\documentclass[12pt, a4paper]{amsart}

\usepackage{latexsym, amssymb,amsmath, amsthm,amsfonts}
\usepackage{hyperref,url}

\newcommand{\kk}{\Bbbk}
\newcommand{\kv}{{\kk[V]}}

\newcommand{\kvg}{{\kk[V]^G}}

\def\SL2{\operatorname{SL}_{2}(K)}

\def\GL2{\operatorname{GL}_{2}(K)}
\def\Ga{{\mathbb G}_{a}}

\def\INVSL2{$K[V]^{operatorname{SL}_{2}(K)}$}
\def\INVSO2{$K[V]^{operatorname{SO}_{2}(K)}$}
\def\INVGL2{$K[V]^{operatorname{GL}_{2}(K)}$}

\def\GL{\operatorname{GL}}
\def\SL{\operatorname{SL}}

\def\myforall{\textrm{ for all }}

\newtheorem{Lemma}{Lemma}[section]
\newtheorem{Theorem}[Lemma]{Theorem}

\newtheorem{Prop}[Lemma]{Proposition}

\theoremstyle{definition}

\theoremstyle{remark}
  \newtheorem{rem}[Lemma]{Remark}

\newtheoremstyle{Acknowledgments}
  {}
    {}
     {}
     {}
    {\bfseries}
    {}
     {.5em}
     {\thmname{#1}\thmnumber{ }\thmnote{ (#3)}}
\theoremstyle{Acknowledgments}
\newtheorem{ack}{Acknowledgments.}

\title[An Explicit Finite Separating Set]
{A finite separating set for Daigle and Freudenburg's counterexample to
  Hilbert's Fourteenth Problem}

\author{ Emilie Dufresne}
\address{Mathematics Center Heidelberg (MATCH) \\
Ruprecht-Karls-Uni\-ver\-si\-t\"at Heidelberg\\
Im Neuenheimer Feld 368\\
69120 Heidelberg, Germany}
\email{emilie.dufresne@iwr.uni-heidelberg.de}

\author{ Martin Kohls}
\address{Technische Universit\"at M\"unchen \\
 Zentrum Mathematik-M11\\
Boltzmannstrasse 3\\
 85748 Garching, Germany}
\email{kohls@ma.tum.de}

\date{\today}

\begin{document}


\begin{abstract}{ This paper gives the first explicit example of a finite
    separating set in an invariant ring which is not finitely generated, namely, for Daigle and Freudenburg's
   5-dimensional counterexample to Hilbert's Fourteenth Problem.
}\end{abstract}


\maketitle

\section{Introduction}
Hilbert's Fourteenth Problem asks if the ring of invariants of an algebraic group action on an affine variety is always finitely
generated. The answer is negative in general: Nagata \cite{NagataHilbert14}
gave the first counterexample in 1959. In characteristic zero, the
Maurer-Weitzenb\"ock Theorem~\cite{Weitzenboeck} tells us that linear actions
of the additive group have finitely generated invariants, but nonlinear
actions need not have finitely generated invariants. Indeed, there are several
such examples, the smallest being Daigle and Freudenburg's 5-dimensional
counterexample \cite{DaigleFreudenburg} to Hilbert's Fourteenth Problem.

Although rings of invariants are not always finitely generated, there always exists a finite  separating set \cite[Theorem 2.3.15]{DerksenKemper}. In other words, if $\kk$ is a field and if a group $G$ acts
 on a finite dimensional $\kk$-vector space $V$, then there always exists a
 finite subset $E$ of the invariant ring $\kvg$ such that if, for two points $x,y\in V$, we have $f(x)=f(y)$ for all $f\in E$, then $f(x)=f(y)$ for all $f\in
\kvg$. This notion was introduced by Derksen and Kemper \cite[Section
 2.3]{DerksenKemper}, and has gained a lot of attention in the recent years,
 for example see \cite{DomokosSep,DraismaSeparating,DufresneSeparating, DufresneElmerKohls,KemperSeparating,SezerSeparating}.

 The proof of the existence of a finite separating set is not
 constructive, and until now, no example was known for infinitely generated invariant rings. The main result of this paper
 is to give the first example: a finite separating set for  Daigle and Freudenburg's 5-dimensional counterexample to Hilbert's Fourteenth Problem.

\begin{ack}
We thank Gregor Kemper for funding visits of the first author to TU M\"unchen.
\end{ack}


\section{Daigle and Freudenburg's Counterexample}

We now introduce the notation used throughout the paper, and set up the example. We recommend the book of Freudenburg \cite{FreudenburgBook} as an excellent reference for locally nilpotent derivations.

Let $\kk$ be a field of characteristic zero, and let $\Ga$ be its additive group. If $V=\kk^{5}$ is a 5-dimensional vector space over $\kk$, then $\kv=\kk[x,s,t,u,v]$ is a polynomial ring in five
variables. Daigle and Freudenburg \cite{DaigleFreudenburg} define a locally
nilpotent $\kk$-derivation on $\kv$:
\[
D=x^{3}\frac{\partial}{\partial s} + s\frac{\partial}{\partial t} +
t\frac{\partial}{\partial u} + x^{2}\frac{\partial}{\partial v}.
\]
This derivation $D$ induces an action of $\Ga$ on $V$. If $r$ is an additional
indeterminate, then the corresponding map of $\kk$-algebras is
\begin{equation}\label{muDefinition}
\mu: \kv\rightarrow \kv[r],\quad f\mapsto
\mu(f)=\mu_{r}(f)=\sum_{k=0}^{\infty}\frac{D^{k}(f)}{k!}r^{k},
\end{equation}
where $\kv[r]\cong\kv\otimes_{\kk}\kk[\Ga]$. The induced action of  $\Ga$ on $\kv$ is:
\[
(-a)\cdot f:=\mu_{a}(f)\quad \myforall a\in \Ga, f\in \kv.
\]
In particular, for $a\in\Ga$, we have
\[
(-a)\cdot x=x, \quad (-a)\cdot s=s+ax^{3}, \quad (-a)\cdot t=t+as+\frac{a^{2}}{2}x^{3},\]\[
(-a)\cdot u=u+at+\frac{a^{2}}{2}s+\frac{a^{3}}{6}x^{3}, \quad (-a)\cdot v=v+ax^{2}.
\]
The invariant ring $\kv^{\Ga}$ coincides with the kernel of $D$. Define a grading on $\kv$ by assigning $\deg x = 1, \quad \deg s=\deg t= \deg u = 3,\quad \deg v=2$. As the action of $\Ga$ on $\kv$ and the derivation $D$ are homogeneous with respect to this grading, the ring of invariants is a graded subalgebra. We write $\kv^{\Ga}_+$ to denote the unique maximal homogeneous ideal of~$\kv^{\Ga}$.
Daigle and Freudenburg \cite{DaigleFreudenburg} proved that $\kv^{\Ga}=\ker D$ is not
finitely generated as a $\kk$-algebra. The main result of this paper is to exhibit a finite geometric separating set.

\begin{Theorem}\label{MainResult}
Let $\Ga$ act on $V$ as above. The following 6 homogeneous polynomials are invariants and form a separating set $E$ in $\kv^{\Ga}$:
\[
f_{1}=x,\quad f_{2}=2x^{3}t-s^{2}, \quad f_{3}=3x^{6}u-3x^{3}ts+s^{3},
\]
\[
f_{4}=xv-s,  \quad f_{5}=x^{2}ts-s^{2}v+2x^{3}tv-3x^{5}u,
\]
\[
f_{6}=-18x^{3}tsu+9x^{6}u^{2}+8x^{3}t^{3}+6s^{3}u-3t^{2}s^{2}.
\]
\end{Theorem}

\begin{rem}
In \cite[Lemma 12]{Winkelmann}, Winkelmann shows that these six invariants
 separate orbits outside  $\{p\in V: x(p)=s(p)=0\}$, which as we will see
 later, is the easy case. (Note that in \cite{Winkelmann} there is a typo in the invariant
  we denoted by $f_{6}$.)
\end{rem}


\section{Proof of Theorem \ref{MainResult}}
In this section, we prove our main result. We start by establishing some useful facts.

\begin{Lemma}\label{Localizationinvariants}
$\kv^{\Ga}\subseteq \kk[f_{1},f_{2},f_{3},f_{4},\frac{1}{x}]$.
\begin{proof}
As $x$ is a constant, the derivation $D$ extends naturally to
$\kv_{x}$ via $D\left(\frac{f}{x^{n}}\right):=\frac{D(f)}{x^{n}}$ $\myforall f\in\kv, n\in{\mathbb N}$,
and we have $\kv^{\Ga}\subseteq (\kv_{x})^{\Ga}$. The element $\frac{s}{x^{3}}\in \kv_{x}$ satisfies
  $D\left(\frac{s}{x^{3}}\right)=1$, that is, it is a slice. By the Slice
  Theorem (see \cite[Proposition 2.1]{VanDenEssen}, or \cite[Corollary 1.22]{FreudenburgBook}), we obtain a generating set of the invariant ring
  $\kv_{x}^{\Ga}$ by applying $\mu$ to the generators of $\kv_{x}=\kk[x,s,t,u,v,\frac{1}{x}]$ and ``evaluating'' at $r=-\frac{s}{x^{3}}$. Therefore, we have
\begin{eqnarray*}
\kv_{x}^{\Ga}&=&\kk\left[\mu_{-\frac{s}{x^{3}}}(x),\mu_{-\frac{s}{x^{3}}}(s),\mu_{-\frac{s}{x^{3}}}(t),\mu_{-\frac{s}{x^{3}}}(u),\mu_{-\frac{s}{x^{3}}}(v),\mu_{-\frac{s}{x^{3}}}(\tfrac{1}{x})\right]\\
&=&\kk\left[f_{1},0,\frac{f_{2}}{2x^{3}},\frac{f_{3}}{3x^{6}},\frac{f_{4}}{x},\frac{1}{x}\right]. 
\end{eqnarray*}\end{proof}
\end{Lemma}


\begin{proof}[Proof of Theorem \ref{MainResult}.]
 First, note that $f_{i}$ is invariant for $i=1,\ldots,6$. Let
 $p_{i}=(\chi_{i},\sigma_{i},\tau_{i},\omega_{i},\nu_{i})$, $i=1,2$, be two
 points in $V$ such that $f_{i}(p_{1})=f_{i}(p_{2})$ for each
 $i=1,\ldots,6$. We will show that $f(p_{1})=f(p_{2})$ for all
 $f\in\kv^{\Ga}$. Since $f_{1}=x$, we have $\chi_{1}=\chi_{2}$. If
 $\chi_{1}=\chi_{2}\ne 0$, then Lemma \ref{Localizationinvariants} implies $f(p_{1})=f(p_{2})$ for all $f\in\kv^{\Ga}$. Thus, we may assume $\chi_{1}=\chi_{2}=0$. It follows that $\sigma_{1}=-f_{4}(p_{1})= -f_{4}(p_{2})=\sigma_{2}$.
Define a linear map
\[
\gamma: \kk^{5}\rightarrow \kk^{4},\quad
(\chi,\sigma,\tau,\omega,\nu)\mapsto(\sigma,\tau,\omega,\nu),
\]
and a $\kk$-algebra morphism
\[
\rho: \kk[x,s,t,u,v]\rightarrow \kk[s,t,u,v], \quad f(x,s,t,u,v)\mapsto
f(0,s,t,u,v).
\]
Define  a $\kk$-linear locally nilpotent derivation on $\kk[s,t,u,v]$ via
\[\Delta= s\frac{\partial}{\partial t} +
t\frac{\partial}{\partial u}.
\]
One easily verifies that $\Delta\circ\rho=\rho\circ D$. In particular, $\rho$ induces a map $\ker D\rightarrow \ker \Delta$. The kernel of $\Delta$ is known (or can be computed with van den Essen's Algorithm \cite{VanDenEssen}):
it corresponds to the binary forms of degree $2$, that is,
\begin{equation}\label{KernDelta}
\ker \Delta=\kk[s,2us-t^{2},v].
\end{equation} 
Since $\chi_{i}=0$,  we have
$f(p_{i})=\rho(f)(\gamma(p_{i}))$ for $i=1,2$ and any $f\in \kv$. Thus, to show $f(p_{1})=f(p_{2})$ for all $f\in \kk[V]^{\Ga}=\ker D$, it suffices to show
$f(\gamma(p_{1}))=f(\gamma(p_{2}))$ for all $f \in \rho(\ker D)\subseteq\kk[s,2us-t^{2},v]$.

If $\sigma_{1}=\sigma_{2}\ne 0$, then the values of $s,2us-t^{2},v$ on
$\gamma(p_{i})$ are uniquely determined by the values of
$\rho(f_{4})=-s, \,\,\rho(f_{5})=-s^{2}v$, and
$\rho(f_{6})=3s^{2}(2us-t^{2})$ on $\gamma(p_{i})$ for $i=1,2$. Since
\[\rho(f_{i})(\gamma(p_{1}))=f_{i}(p_{1})=f_{i}(p_{2})=\rho(f_{i})(\gamma(p_{2}))
\quad \myforall i=1,\ldots,6,
\] the case $\sigma_{1}=\sigma_{2}\ne 0$ is done. Assume
$\chi_{1}=\chi_{2}=\sigma_{1}=\sigma_{2}=0$, then by Proposition \ref{TheMainProposition}, $f(p_{1})=f(p_{2})=f(0,0,0,0,0)$ for all $f\in \kv^{\Ga}$.
\end{proof}


\begin{Prop}\label{TheMainProposition}
We have $\kv^{\Ga}\subseteq \kk \oplus (x,s)\kv$.
\end{Prop}

This proposition is the key to the proof of Theorem
\ref{MainResult}. It could be obtained from a careful study of the generating
set of $\kv^{\Ga}$ given by Tanimoto \cite{Tanimoto}. We give a more
self-contained proof, which relies only on the van den Essen-Maubach
Kernel-check Algorithm (see \cite{VanDenEssen}, and \cite[p. 32]{MaubachPhD}). 

\begin{proof}[Proof of Proposition \ref{TheMainProposition}]
It suffices to show that $\kv^{\Ga}_{+}\subseteq (x,s)\kv$. By way of contradiction, suppose there exists $f\in
\kv^{\Ga}_{+}$ of the form $f=xp+sq+h(t,u,v)$,  where $p,q\in\kv$,  and $h(t,u,v)\ne 0$.

Without loss of generality, we can assume $f$ is homogeneous of positive degree. We apply the map
$\rho$ from the proof of Theorem~\ref{MainResult}. By Equation \eqref{KernDelta}, we have $f(0,s,t,u,v)\in \kk[s,2us-t^{2},v]$, so we have $f(0,0,t,u,v)=h(t,u,v)\in\kk[0,-t^{2},v]$, and we set $h(t,v):=h(t,u,v)\in\kk[t,v]$. Since $f$ is homogeneous, so is $h$, and there is a unique monomial $t^{d}v^{e}$ in $h$ such that the exponent $e$ of $v$ is maximal. Clearly, $D\circ \frac{\partial}{\partial v}=\frac{\partial}{\partial v}\circ D$, and so, for all $k$, we have
\[
\frac{\partial^{k}f}{\partial v^{k}}=x\frac{\partial^{k}p}{\partial
  v^{k}}+s\frac{\partial^{k}q}{\partial
  v^{k}}+\frac{\partial^{k}h(t,v)}{\partial v^{k}} \in \kk[V]^{\Ga}.
\]
If $d=0$, then taking $k=e-1$, implies $v$ is the only  monomial  appearing in
$\frac{\partial^{e-1}h(t,v)}{\partial v^{e-1}}$ (since $v$ has degree $2$, and $t$ has degree $3$, $t$ cannot have nonzero exponent). Thus, there is a homogeneous invariant of degree~$2$ of the form
$x\tilde{p}+s\tilde{q}+v\in \kk[V]^{\Ga}$, but as $x^{2}$ spans the space of invariants of degree $2$, we have a contradiction.

Assume now that $d>0$. If $k=e$, then $t^{d}$ is the only monomial
appearing in $\frac{\partial^{e}h(t,v)}{\partial v^{e}}$. Thus, replacing $f$ by
$\frac{\partial^{e}f}{\partial v^{e}}$, and dividing by the coefficient of
$t^{d}$, we can assume $f=xp+sq+t^{d}$, where $p,q\in\kv$  and  $d>0$. Since $f(x,s,t,u,v)\in \ker D$, Lemma~\ref{TheHDerivation}~(a) implies the element 
\begin{eqnarray}\label{TheContradictingElement}
g(x,t,u,v)&:=&f(x,xv,t,u,v)\nonumber\\
&=&x\tilde{p}+xv \tilde{q}+t^{d} \in \kk[x,t,u,v]
\end{eqnarray}
lies in the kernel of the derivation $\Delta':=x^{2}\frac{\partial}{\partial
  v}+xv\frac{\partial}{\partial t}+t\frac{\partial}{\partial u}$ defined on $\kk[x,t,u,v]$.
As no monomial of the form $t^{k}$ (with $k>0$) appears in the
  four generators of $\ker \Delta'$ (by Lemma \ref{TheHDerivation}~(b)), the
  monomial $t^{d}$ cannot appear as a monomial in $g\in
  \ker \Delta'$, and so we have a contradiction.
\end{proof}

In the following Lemma, we write $\kk[x,v,t,u]$ rather than
$\kk[x,t,u,v]$, so that the derivation $\Delta'$ is triangular.

\begin{Lemma}\label{TheHDerivation}
Define a $\kk$-algebra map
\begin{eqnarray*}
\phi:\kk[x,s,t,u,v]&\rightarrow& \kk[x,v,t,u],\\
 f(x,s,t,u,v)&\mapsto &\phi(f)(x,v,t,u):=
f(x,xv,t,u,v),\end{eqnarray*}
and a derivation $\Delta'$ on $\kk[x,v,t,u]$:
\[
 \Delta'=x^{2}\frac{\partial }{\partial v}+xv\frac{\partial }{\partial
 t}+t\frac{\partial }{\partial u}.
\]
It follows that
\begin{enumerate}
\item[(a)]  $\Delta'\circ\phi=\phi\circ D$, in particular, $\phi$ maps $\ker D$
  to $\ker \Delta'$;

\item[(b)]  $\ker \Delta'=\kk[h_{1},h_{2},h_{3},h_{4}]$, where
\[
h_{1}=x, \quad h_{2}=2xt-v^{2}, \quad h_{3}=3x^{3}u-3xvt+v^{3},
\]
\begin{eqnarray*}
h_{4}&=&8xt^{3}+9x^{4}u^{2}-18x^{2}tuv-3t^{2}v^{2}+6xuv^{3}\\
&=&(h_{2}^{3}+h_{3}^{2})/x^{2}.
\end{eqnarray*}\end{enumerate}
\begin{proof}
(a): For $f=f(x,s,t,u,v)\in \kk[x,s,t,u,v]$, we have
\begin{eqnarray*}
(\Delta'\circ \phi) (f)&=& (x^{2}\frac{\partial }{\partial v}+xv\frac{\partial }{\partial
 t}+t\frac{\partial }{\partial u})f(x,xv,t,u,v)\\
&=&x^{3}\phi(\frac{\partial f}{\partial s})+x^{2}\phi(\frac{\partial f}{\partial v})+xv\phi(\frac{\partial f}{\partial
 t})+t\phi(\frac{\partial f}{\partial u})\\
&=&\phi\left(  x^{3}\frac{\partial f}{\partial s}+x^{2}\frac{\partial f}{\partial v}+s\frac{\partial f}{\partial
 t}+t\frac{\partial f}{\partial u}\right)=(\phi \circ D) (f).
\end{eqnarray*}
(b): Since $\Delta'$ is a triangular monomial derivation of a four dimensional
polynomial ring, by Maubach \cite{Maubach},
its kernel is generated by at most four elements. In fact, \cite[Theorem 3.2,
Case 3]{Maubach} gives the same generators for $\ker \Delta'$, up to
multiplication by a scalar (the formula for $h_{4}$ contains a typo).

Alternatively, one can use van den Essen's Algorithm. As in the proof of
Lemma \ref{Localizationinvariants}, the derivation $\Delta'$ can be extended to $\kk[x,v,t,u]_{x}$,
and as $\Delta'(\frac{v}{x^{2}})=1$, the Slice Theorem \cite[Proposition 2]{VanDenEssen}
yields
\begin{equation}\label{kerHlocal}
(\ker
\Delta')_{x}=\mu_{-\frac{v}{x^{2}}}\left(\kk[x,v,t,u,\tfrac{1}{x}]\right)=\kk[h_{1},h_{2},h_{3},\tfrac{1}{x}],
\end{equation}
where $\mu$ is defined similarly as in Equation \eqref{muDefinition}. Consider the additional invariant $h_{4}:=(h_{2}^{3}+h_{3}^{2})/x^{2}\in
\kk[x,v,t,u]$. We claim $\ker \Delta'=\kk[h_{1},h_{2},h_{3},h_{4}]=:R$. Equation
\eqref{kerHlocal} implies $R\subseteq \ker \Delta'\subseteq R_{x}$. Next, we look at the ideal
of relations modulo $x$ between the generators of $R$,
\begin{eqnarray*}
I&:=&\{P\in \kk[X_{1},X_{2},X_{3},X_{4}]\mid P(h_{1},h_{2},h_{3},h_{4})\in
(x)\kk[x,v,t,u]\}\\
&=&\{P\in \kk[X_{1},X_{2},X_{3},X_{4}]\mid P(0,-v^{2},v^{3},-3t^{2}v^{2})=0\}\\
&=&(X_{1},X_{2}^{3}+X_{3}^{2})\kk[X_{1},X_{2},X_{3},X_{4}].
\end{eqnarray*}
Since $(h_{2}^{3}+h_{3}^{2})/x=xh_{4}\in R$, we have that
$P(f_{1},f_{2},f_{3},f_{4})/x\in R$ for every $P\in I$, and  the Kernel-check
algorithm implies $\ker \Delta'=R$ (see \cite[p. 184]{FreudenburgBook}).
\end{proof}
\end{Lemma}


\bibliographystyle{plain}
\bibliography{MyBib}
\end{document}